\documentclass[12pt]{amsart}
\usepackage{amssymb,amsmath,graphicx,verbatim,eufrak,amsthm}
\usepackage{color}
\usepackage{xcolor}
\usepackage{mdframed}
\usepackage{bbm}
\usepackage{hyperref}
\hypersetup{
    linktoc=page,  
    colorlinks=true,       
    linkcolor=red,          
    citecolor=green,        
    filecolor=magenta,      
    urlcolor=cyan           
}

\definecolor{env_back}{gray}{0.8}
\definecolor{thm_color}{rgb}{0,0,1}
\definecolor{conj_color}{rgb}{1,0,0}
\definecolor{dfn_color}{cmyk}{0,1,0,0}

\newtheorem{thm}{Theorem}[section]
\newtheorem{cor}[thm]{Corollary}
\newtheorem{lem}[thm]{Lemma}

\newtheorem{prop}[thm]{Proposition}

\newtheorem{ques}[thm]{Question}
\newtheorem{conj}[thm]{Conjecture}

\newtheorem*{clm*}{Claim}

\theoremstyle{definition}
\newtheorem{dfn}[thm]{Definition}
\newtheorem{exm1}[thm]{Example}

\theoremstyle{remark}
\newtheorem{rem}[thm]{Remark}
\newtheorem*{rem*}{Remark}

\newenvironment{lem*}[1]{\vspace{1ex}\noindent
{\bf Lemma* (#1).} [restatement]  \hspace{0.5em} \em }{ }
\newenvironment{thm*}[1]{\vspace{1ex}\noindent
{\bf Theorem* (#1).} [restatement]  \hspace{0.5em} \em }{ }

{\begin{exm1}}%
{\hfill$\bigtriangleup\bigtriangledown\bigtriangleup$
\end{exm1} }

\newcommand{\set}[1]{\left\{#1\right\}}

\newcommand{\sr}[1]{\left(#1\right)}

\newcommand{\Integer}{\mathbb{Z}}

\newcommand{\Z}{\Integer}

\newcommand{\eps}{\varepsilon}

\newcommand{\ie}{{\em i.e.\ }}
\newcommand{\eg}{{\em e.g.\ }}

\DeclareMathOperator{\E}{\mathbb{E}}     

\renewcommand{\Pr}{}
\let\Pr\relax
\DeclareMathOperator{\Pr}{\mathbb{P}}

\newcommand{\1}[1]{\mathbf{1}_{\set{ #1 } }}

\def\squareforqed{\hbox{\rlap{$\sqcap$}$\sqcup$}}
\def\qed{\ifmmode\squareforqed\else{\unskip\nobreak\hfil
\penalty50\hskip1em\null\nobreak\hfil\squareforqed
\parfillskip=0pt\finalhyphendemerits=0\endgraf}\fi}

\newcommand{\ignore}[1]{ }

\newcommand{\p}{\partial}

\newcommand{\dist}{\mathrm{dist}}
\newcommand{\vphi}{\varphi}

\newcommand{\define}[1]{{\bf #1}}

\newcommand{\rad}{\mathrm{rad}}

\usepackage{capt-of}

\newcommand{\capac}{\mathrm{cap}}

\newcommand{\isodim}{\mathrm{dim_{iso}}}

\newcommand{\arXiv}[1]{\href{http://arxiv.org/abs/#1}{\texttt{arXiv:#1}} }

\begin{document}

\title[Upper bounds on DLA]{Upper bounds on the growth rate of Diffusion Limited Aggregation}

\author{Itai Benjamini}
\address{IB: Weizmann Institute of Science}
\email{itai.benjamini@weizmann.ac.il}

\author{Ariel Yadin}
\address{AY: Ben-Gurion University of the Negev}
\email{yadina@bgu.ac.il}

\thanks{AY supported by the Israel Science Foundation (grant no.\ 1346/15).}

\begin{abstract}
We revisit Kesten's argument for the upper bound on the growth rate
of DLA.  We are able to make the argument robust enough so that
it applies to many graphs, where only control of
the heat kernel is required.  We apply this to many examples including
transitive graphs of polynomial growth, graphs of exponential growth,
non-amenable graphs, super-critical percolation on $\Z^d$, high dimensional
pre-Sierpinski carpets.
We also observe that a careful analysis shows that Kesten's original bound on
$\Z^3$ can be improved from $t^{2/3}$ to $\sqrt{t \log t}$.
\end{abstract}

\maketitle

\section{Introduction}

In this note we study DLA on general graphs, especially in the context of
some control on the decay of the {\em heat kernel} of the random walk.
DLA process defined by Witten \& Sander \cite{WittenSander} in 1981, 
is a notoriously resilient to rigourous analysis.
It is traditionally studied on the Euclidean graphs $\Z^d$.

In the DLA model, we aggregate particles on a graph as follows
(for a precise definition see Definition \ref{dfn:DLA} below).
Start with a particle sitting in the graph $G$.  Given the current aggregate
after $t$ particles, let a new particle move randomly on the graph
``coming in from infinity'', and let that particle stick to the aggregate when it first reaches the boundary
of the aggregate, thus forming an aggregate with one more particle.

Simulations show that DLA tends to have fractal structure, and $0$ density in the long run.
However, there are no rigorous mathematical arguments that we are aware of that prove this fact,
and we do not contribute something in this direction either.
For this reason variants of DLA have been studied, see \eg
\cite{BPP97, BYcylinder, BY16, Eldan, FP17}.

Perhaps the most notable (and almost the only) theorem regarding the standard version of
DLA is that of Kesten \cite{Kesten1, Kesten2, Kesten3},
which gives an upper bound on the growth rate of the DLA aggregate.
Kesten shows that in $\Z^d$ the DLA aggregate after $t$ particles reaches a distance of at most
$t^{2/d}$ when $d \geq 3$, and $t^{2/3}$ when $d=2$.

In this paper we ``clean up'' Kesten's argument, and make it more robust, so that basically only
good control on the decay of the heat kernel is required to apply it.
We are thus able to generalize Kesten's results to many more cases, including
transitive graphs of polynomial growth, non-amenable graphs, graphs of exponential growth, and
fractal graphs.

We also observe that properly applying Kesten's method improves the bound in the $\Z^3$ case
to an upper bound of $(t \log t)^{1/2}$ on the growth rate of the aggregate of $t$ particles.

In Section \ref{scn:examples} we provide many examples to which our methods apply.
After introducing some notation and properly defining DLA,
we collect the results in the paper in Theorem \ref{thm:main} below,
each with a reference to its corresponding specific argument
in Section \ref{scn:examples}.

\subsection{Notation}

Throughout we will always consider infinite, simple, connected graphs.
For a graph $G$ and some fixed vertex $o \in G$ we say that $G$ is rooted at $o$.
We write $|x| = \dist(x,o)$.  If $A \subset G$ is some finite set we write
$\rad(A) = \max \{ |x| \ : \ x \in A\}$.
The (outer) boundary of $A$ is defined to be $\p A = \{ x \not\in A \ : \ \exists \ y \in A \ , \ y \sim x \}$.

$(X_t)_t$ denotes the random walk on $G$, where
$\Pr_x, \E_x$ denote
the corresponding probability measure
and expectation conditioned on $X_0=x$.  When $x=o$ we may omit the subscript.
Define the stopping times
$$ T_A = \inf \set{ t \geq 0 \ : \ X_t \in A } \qquad T_A^+ = \inf \set{ t \geq 1 \ : \ X_t \in A } , $$
and $T_x = T_{\set{x}} , T_x^+ = T_{\set{x}}^+$.

We will always assume throughout that $G$ is transient; \ie $\Pr [ T_o^+ = \infty ] > 0$.
The \define{heat kernel} is defined to be
$$ p_t(x,y) = \Pr_x [ X_t = y ] . $$
The \define{Green function} is defined by
$$ g(x,y) = \sum_{t \geq 0} p_t(x,y) . $$
This is well defined and finite precisely when $G$ is transient.
It is well known that $\deg(x) g(x,y) = \deg(y) g(y,x)$,
and that $g$ is harmonic in the first variable, except for at $x=y$;
that is $\Delta g(\cdot , y) = \delta_y(\cdot)$
(here $\Delta = I-P$ is the Laplacian, $P$ is the transition matrix
of the random walk).

For a finite set $A \subset G$ we define the \define{equilibrium measure} on $A$ as
$$ e_A(x) := \1{ x \in A } \cdot \deg(x) \cdot \Pr_x [ T_A^+ = \infty ]  $$
and the \define{capacity} of $A$ as
$$ \capac(A) : = \sum_x  e_A(x) . $$

For a finite set $A \subset G$ we define the \define{harmonic measure} (from infinity)
of $A$ as the probability measure
$$ h_A(x) = \frac{ e_A(x) }{ \capac(A) } . $$
A well known estimate on capacity (see \eg \cite{ARS14}) is:
\begin{align}
\label{eqn:capac}
\inf_{x \in A} \sum_{y \in A} g(x,y) & \leq
\frac{ \sum_{x \in A} \deg(x)  }{ \capac(A) }
\leq \sup_{x \in A} \sum_{y \in A} g(x,y) .
\end{align}

\subsection{DLA}

{\em Diffusion Limited Aggregation}, or {\em DLA}, is a random process
introduced by Witten \& Sander
\cite{WittenSander} in 1981.  
The definition of DLA is as follows.

\begin{dfn} \label{dfn:DLA}
Let $G$ be a graph. DLA is a Markov chain on finite subsets of $G$, denoted $(A_t)_t$,
evolving according to the following law.

Start with some finite set $A_0 = A$.
Given $A_t$, the process at time $t$, let $a_{t+1}$ be a random vertex in $\p A_t$
chosen according to the harmonic measure of $A_t \cup \p A_t$.  That is,
$$ \Pr [ a_{t+1} = x \ | \ A_t ] : = h_{A_t \cup \p A_t} (x) . $$
Then, set $A_{t+1} = A_t \cup \{ a_{t+1} \}$.
\end{dfn}

It is immediate that $|A_t|=|A_0|+t$.
Thus, if the graph $G$ is $\Z^d$, then
$\rad(A_t) \geq c t^{1/d}$ for some constant $c>0$ and all $t$.
Simulations indicate that this is far from what actually occurs.
In fact, the aggregate $A_t$ seems to be fractal, with a non-trivial Hausdorff dimension.

\begin{conj}
Let $G= \Z^d$.  Let $(A_t)_t$ be DLA on $G$ started at $A_0 = \{0\}$.
Then, there exists $\tfrac1d < \alpha < 1$ such that
$$ 0 < \liminf_{t \to \infty} t^{-\alpha} \cdot \E \rad(A_t)
\leq \limsup_{t \to \infty} t^{-\alpha} \cdot \E \rad(A_t) < \infty . $$
\end{conj}

This conjecture has been open for almost 40 years, and basically there has been no advancement
towards any lower bound for such an $\alpha > \tfrac1d$.
As for the upper bound, this is one of the only known mathematical result, due to Kesten
\cite{Kesten ????}, which states that for DLA on $\Z^d$ we have that a.s.\
$$ \limsup_{t \to \infty} t^{-\alpha} \cdot \rad(A_t) < \infty , $$
with
$\alpha = \tfrac23$ for $d=2$ and $\alpha = \tfrac2d$ for $d \geq 3$.

Even the following naive conjecture seems wide open.

\begin{conj}
Let $G= \Z^d$.  Let $(A_t)_t$ be DLA on $G$ started at $A_0 = \{0\}$.
Then,
$$ \limsup_{t \to \infty} t^{-1/d} \E \rad(A_t) = 0 . $$
\end{conj}

\subsection{Main results}

We now collect the bounds we provide in this paper.
Next to each result we provide a reference to the corresponding argument in Section \ref{scn:examples}.

\begin{thm} \label{thm:main}
Let $G$ be a bounded degree transient graph rooted at $o$.
Let $(A_t)_t$ be DLA on $G$ started at $A_0=\{o\}$.
Under the following conditions, we have that a.s.\
$ \limsup_{t \to \infty}  \tfrac{ \rad(A_t) }{ f(t) } < \infty , $
where:

\vspace{20pt} \noindent
\begin{tabular}{ | p{6cm} | c | c  |}
\hline
graph condition & $f(t)$ & reference \\
\hline \hline
transitive of polynomial growth (degree $d \geq 4$) &
$t^{2/d}$ &
Theorem \ref{thm:poly-growth} \\
\hline
$\Z^3$ or any transitive of cubic growth &
$\sqrt{t \log t}$ &
Theorem \ref{thm:poly-growth} \\
\hline
transitive and exponential growth &
$(\log t)^4$ &
Theorem \ref{thm:exp growth} \\
\hline
pinched exponential growth &
$(\log t)^4$ &
Theorem \ref{thm:exp growth} \\
\hline
non-amenable &
$\log t$ &
Theorem \ref{thm:non amen DLA}
\\ \hline
transitive and has super polynomial growth &
$t^\eps$ for any $\eps>0$ &
Theorem \ref{thm:super-poly growth} \\
\hline
super-critical percolation on $\Z^d , d \geq 3$ &
as in $\Z^d$ &
Theorem \ref{thm:super-crit perc} \\
\hline
$3$-dimensional pre-Sierpinski carpet &
$t^\beta$ for $\beta = \tfrac{ \log_2 (13)  - 2}{3} = 0.5568\ldots$ &
Theorem \ref{thm:Sierpinski} \\
\hline
$4$-dimensional pre-Sierpinski carpet &
$t^{1/2}$ &
Theorem \ref{thm:Sierpinski} \\
\hline
$n$-dimensional pre-Sierpinski carpet, $n \geq 5$ &
$t^\beta \log t$ for
$\beta = \tfrac{ \log (3^n-1) - \log (3^{n-1}-1) }{ \log (3^{n-1}-1)   - \log \sqrt{ 3^n-1}  }$ &
Theorem \ref{thm:Sierpinski} \\
\hline
\end{tabular}
\end{thm}

\newpage

\section{Kesten's method}

\subsection{Beurling estimates}

A key step in Kesten's method is proving Beurling-type estimates for general sets.

\begin{dfn}
Let $G$ be a graph rooted at $o$.
Let $\vphi : [0,\infty) \to [0,1]$ be a non-increasing function with $\lim_{x \to \infty} \vphi(x) = 0$.

We say that $(G,o)$ has a \define{$\vphi$-volume Beurling estimate}
if for any connected finite set $o \in A \subset G$
we have that
$$ \sup_{x \in A} h_A(x) \leq \vphi(|A|) . $$

We say that $(G,o)$ has a \define{$\vphi$-radius Beurling estimate}
if for any connected finite set $o \in A \subset G$
we have that for any $x$
$$ \sup_{x \in A} h_A(x) \leq \vphi(\rad(A)) . $$
\end{dfn}

\begin{lem} \label{lem:kesten}
Let $G$ be a graph rooted at $o$ of bounded degree $\sup_x \deg(x) \leq D$.
Let $(A_t)_t$ be DLA on $G$ started at $A_0 = \set{o}$.
Suppose that $G$ has a $\vphi$-volume Beurling estimate.
Then, for any $t >0, s \geq n>0$ we have
\begin{align*}
\Pr [ \rad(A_{s+t})  \geq \rad(A_s) + n \ | \ A_s ]
& \leq
s \cdot \exp \Big( n \cdot \log  \big( \frac{De}{n} \sum_{j=0}^{t-1} \vphi( s+j ) \big)  \Big) .
\end{align*}
\end{lem}

First an auxiliary large deviations calculation:
\begin{lem}
\label{lem:LD}
Let $B = \sum_{n=1}^k Z_n$ for independent Bernoulli random variables $(Z_n)_n$,
each of mean $\E Z_n = p_n$.  Then, for any $C>1$ we have
$$ \Pr [ B \geq C \E B ] \leq \exp \sr{ - \E B \cdot C \log (C/e)  } . $$
\end{lem}

\begin{proof}
We use the well known method by Bernstein.
For $\alpha > 0$ we may bound the exponential moment of $B$ as follows:
\begin{align*}
\E e^{\alpha B} & = \prod_{n=1}^k \E e^{\alpha Z_n} = \prod_{n=1}^k ( (e^{\alpha} -1) p_n + 1 )
\leq \exp \sr{ ( e^{\alpha} -1) \E B } .
\end{align*}
By Markov's inequality,
$$ \Pr [ B \geq C \E B ] = \Pr [ e^{\alpha B} \geq e^{\alpha C \E B} ] \leq \exp \sr{ ( e^{\alpha} - 1) \E B - \alpha C \E B } . $$
So we wish to minimize the term $e^{\alpha} - 1 - \alpha C$ over positive $\alpha$.
Taking derivatives this is minimized when $e^{\alpha} = C$ (recall that $C>1$), so
$$ \Pr [ B \geq C \E B ] \leq \exp \sr{ \E B \cdot (C-1 - C \cdot \log C ) } . $$
\end{proof}

\begin{proof}[Proof of Lemma \ref{lem:kesten}]
Let $A_0 = A$ be some finite starting set of radius $r = \rad(A_0)$.

We say that a sequence of vertices $v_1, v_2, \ldots, v_n$ is {\em filled in order up to time $t$}, if
there exist times $0 < t_1 < t_2 < \cdots < t_n \leq t$ such that
$A_{t_j} = A_{t_j-1} \cup \set{ v_j}$.

Fix such a sequence $v_1, \ldots, v_n$.
For any finite subset $S \subset G$ let $v(S)$ be the first vertex in $v_1,\ldots, v_n$
not in $S$; that is,
$v(S) = v_{J(S)}$ where
$$ J(S) = \min \set{ j \ : \ v_1, \ldots, v_{j-1} \in S } . $$
consider the random variables $Z_1, \ldots, Z_j , \ldots$, defined as the indicators
of the events $\set{ A_j= A_{j-1} \cup \set{ v(A_{j-1}) } }$.
That is, $Z_j$ is the indicator of the event that the upcoming vertex from $v_1,\ldots,v_n$
that has not already been added to the cluster, is now added.

When $G$ has a $\vphi$-volume Beurling estimate, then a.s.,
\begin{align} \label{eqn:Zj}
\E [ Z_{j} \ | \ A_0, \ldots, A_{j-1} ] &  \leq \vphi(|A_{ j-1}|) = \vphi(|A|+j-1) .
\end{align}
Thus, $\sum_{j=1}^t Z_j$ is stochastically dominated by $\sum_{j=1}^t W_j$,
where $W_1, \ldots, W_n$ are independent
Bernoulli random variables of mean $\E[W_j] = \vphi(|A|+j-1)$.

Denote
$ I_\vphi(|A|,t) :=
\sum_{j=0}^{t-1} \vphi (|A|+j) = \sum_{j=1}^t \E W_j . $
Now, the event that $v_1, \ldots, v_n$ are filled in order up to time $t$ implies that
$\sum_{j=1}^t Z_j \geq n$.  Thus, 
using Lemma \ref{lem:LD}, we have
that the probability that $v_1, \ldots, v_n$ are filled in order up to time $t$ is bounded by
$$ \exp \big( n \cdot \log ( \tfrac{e}{n }  I_\vphi(|A| , t) ) \big) . $$

Finally, note that if $\rad(A_t) \geq r+ n$ then there must exist a simple path
$v_0, v_1, \ldots, v_n$ in $G$, such that $v_0 \in A$ and $v_1, \ldots, v_n$
are filled in order up to time $t$.
Since there are at most $|A| \cdot D^n$ possibilities for $v_0 , v_1 , \ldots, v_n$,
we obtain that
\begin{align*}
\Pr [ \rad(A_t)  \geq \rad(A_0) + n \ | \ A_0 = A ] 
& \leq
|A| \cdot \exp \Big( - n \cdot \log  \big( \tfrac{n}{D e  I_\vphi(|A| , t) } \big)  \Big) .
\end{align*}
The lemma follows using the Markov property.
\end{proof}

A slight alteration of the above proof provides a similar result for the case where we have a
radius Beurling estimate.

\begin{lem} \label{lem:kesten radius}
Let $G$ be a graph rooted at $o$ of bounded degree $\sup_x \deg(x) \leq D$.
Let $(A_t)_t$ be DLA on $G$ started at $A_0 = \set{o}$.
Suppose that $G$ has a $\vphi$-radius Beurling estimate.
Then, for any $t >0, s \geq n>0$ we have
\begin{align*}
\Pr [ \rad(A_{s+t})  \geq \rad(A_s) + n \ | \ A_s ] & \leq
s \cdot \exp \Big( n \cdot \log  \big( \frac{De}{n} \vphi(\rad(A_s) ) \cdot t  \big)  \Big) .
\end{align*}
\end{lem}

\begin{proof}
Basically, the proof is the same as that of Lemma \ref{lem:kesten},
except that \eqref{eqn:Zj} becomes
\begin{align}
\E [ Z_{j} \ | \ A_0 , \ldots, A_{j-1} ]  & \leq  \vphi(\rad(A_{j-1}) ) \leq \vphi(\rad(A)) .
\end{align}
This leads to a bound on the probability that $v_1, \ldots, v_n$ are filled in order up to time $t$, which is:
$$ \exp \Big( n \cdot \log \big( \tfrac{e}{n}  \vphi(\rad(A)) \cdot t \big) \Big) . $$
Continuing as in the proof of Lemma \ref{lem:kesten},
by summing over all possibilities for $v_1,\ldots,v_n$, we obtain the required bound.
\end{proof}

\subsection{DLA growth upper bounds}

\begin{thm} \label{thm:kesten bound volume}
Let $G$ be a bounded degree
graph rooted at $o$. Let $(A_t)_t$ be DLA on $G$ started at $A_0=\{o\}$.
Suppose that $G$ satisfies a $\vphi$-volume Beurling estimate, with $\vphi$ one of the following:
\begin{enumerate}
\item $\vphi(s) = C (\log s)^{\beta} s^{-\alpha}$ for some $\alpha \in (0,1)$, $\beta \geq 0$,
\item $\vphi(s) = C (\log s)^{\beta} s^{-1}$ for some $\beta \geq 0$,
\end{enumerate}
with $C>0$ some constant.

Then, respectively, we have that:
\begin{enumerate}
\item $\limsup_{t \to \infty}  (\log t)^{-\beta}  t^{\alpha-1} \cdot \rad(A_t)  < \infty$ a.s.
\item $\limsup_{t \to \infty}  (\log t)^{-\beta-1} \cdot \rad(A_t) < \infty$ a.s.
\end{enumerate}
\end{thm}

\begin{proof}
In each of the cases, in order to apply Lemma \ref{lem:kesten}
we need to bound $I_\vphi(s,s+t) := \sum_{j=0}^{t-1} \vphi(s+j)$.
Indeed, by perhaps changing the constant $C$, we have
\begin{enumerate}
\item $I_\vphi(s,s+t) \leq t \cdot C (\log s)^\beta s^{-\alpha}$, \label{i:alpha}
\item $I_\vphi(s,s+t) \leq C (\log (t+s))^{\beta+1}$. \label{i:polylog}
\end{enumerate}

In Case \eqref{i:alpha}, we choose $t= c s$ and $n = \lceil D e^2 C t (\log s)^\beta s^{-\alpha} \rceil
\asymp c' (\log s)^\beta s^{1-\alpha}$, to obtain from Lemma \ref{lem:kesten},
$$ \Pr [ \rad(A_{2s}) \geq \rad(A_s) + n  \ | \ A_s ] \leq s \cdot e^{- n } , $$
and by Borel-Cantelli we have that
$$  \Pr [ \rad(A_{2s}) \geq \rad(A_s)  + c' (\log s)^\beta s^{1-\alpha} \ \textrm{i.o.} ] = 0 . $$
Thus, a.s.\ the sequence $r_s : = \rad(A_{s})$ satisfies
that there exists a (random) constant $K>0$ such that for all $s$ we have
$r_{2s} \leq r_s + K (\log s)^\beta s^{1-\alpha}$.  That is,
a.s.\ there exists $K >0$ such that for all $s$,
$$ r_s \leq c(\alpha) K (\log s)^\beta s^{1-\alpha} \qquad
c(\alpha) = \tfrac{2^{1-\alpha} }{2^{1-\alpha} - 1 } . $$

Case \eqref{i:polylog} is dealt with similarly.
Choose $s=1, n = K (\log (t+1) )^{\beta+1}$ in Lemma \ref{lem:kesten},
for large enough $K = K(C,D)>0$ to obtain that
$$ \Pr [ \rad(A_t) \geq n+1 ] \leq t \cdot e^{- 3 \log (t+1) } \leq t^{-2} . $$
Since this is summable, we are done by Borel-Cantelli.
\end{proof}

\begin{thm} \label{thm:kesten bound radius}
Let $G$ be a bounded degree
graph rooted at $o$. Let $(A_t)_t$ be DLA on $G$ started at $A_0=\{o\}$.
Suppose that $G$ satisfies a $\vphi$-radius Beurling estimate, with
$\vphi(r) = C r^{-\alpha}$ for some $\alpha, C>0$.

Then, a.s.\
$$ \limsup_{t \to \infty} t^{-1/(1+\alpha) } \cdot \rad(A_t) < \infty . $$
\end{thm}

\begin{proof}
As before, the theorem follows from choosing $t,n$ appropriately in Lemma \ref{lem:kesten radius}.
Fix $s$ and set $n = \rad(A_s)$ and $t = \lfloor \frac{n}{D^2 e^2 \vphi(\rad(A_s))} \rfloor$.
Thus, by Lemma \ref{lem:kesten radius},
\begin{align} \label{eqn:radius}
\Pr [ \rad(A_{s+t}) \geq 2 \rad (A_s) \ | \ A_s ] & 
\leq \exp ( - \rad(A_s) ) ,
\end{align}
where we have used that $s = |A_s| \leq D^{\rad(A_s)}$.
When $\vphi(r) = C r^{-\alpha}$ then $t = c \cdot \rad(A_s)^{1+\alpha}$.

Define the following stopping times:
$$ \tau_r : = \inf \set{ s \ : \ \rad(A_s) \geq r } . $$
\eqref{eqn:radius} tells us that for any $r$,
$$ \Pr [ \tau_{2r} \leq \tau_r + c r^{1+\alpha} ] \leq e^{-r} . $$
By Borel-Cantelli, we now have that a.s.\ there exists a (random) constant $K>0$ such that
for all $r$ we have
$\tau_{2r} > \tau_r + K r^{1+\alpha} . $
Since $\tau_r \leq s \iff \rad(A_s) \geq r$, we have
a.s.\ for all $s = \lceil K r^{1+\alpha} \rceil$ that $\rad(A_{s} ) < 2r \leq c(K) s^{1/(1+\alpha)}$.
Thus, a.s.\
$$ \limsup_{s \to \infty} s^{-1/(1+\alpha) } \cdot \rad(A_s) < \infty . $$
\end{proof}

\newpage

\section{Heat kernel decay}

\subsection{Radius Beurling estimates}

\begin{prop} \label{prop:radius Beurling from HK}
Let $G$ be a bounded degree graph rooted at $o$.
Assume that the heat kernel on $G$ satisfies
$$ \sup_{x,y} p_t(x,y) \leq C t^{-d/2}  $$
for some $d>2, C>0$ and all $t > 0$.

Then, $G$ satisfies a $\vphi$-radius Beurling estimate, with
$$ \vphi(r) =
\begin{cases}
\tfrac{C}{3-d}  r^{2-d}  & \textrm{ if } 2<d<3 , \\
C \log r \cdot r^{-1} & \textrm{ if } d = 3 , \\
C r^{-1} & \textrm{ if } d > 3 .
\end{cases}
$$
\end{prop}

\begin{proof}
A well known result of Hebisch \& Saloff-Coste \cite{HSC} tells us that under the
heat kernel condition we have for all $x,y \in G$ and all $t >0$,
$$ p_t(x,y) \leq C t^{-d/2} \cdot \exp( -  \tfrac{\dist(x,y)^2 }{C t} ) , $$
perhaps with a different constant $C>0$.
Summing this over $t$, we obtain a bound on the Green function.
For any $x \neq y \in G$,
$$ g(x,y) \leq C \dist(x,y)^{2-d} , $$
again with a modified constant $C>0$.

This, in turn, implies a radius Beurling estimate for $G$.
Indeed, for any connected set $A$ containing $o$,
if $r = \rad(A)$ then there is a sequence of vertices $o=v_0, v_1, \ldots, v_r$
such that $|v_j|=j$ and $Q : = \{ v_0 , \ldots, v_r \} \subset A$.
Since capacity is monotone, for any $x \in A$,
$$ h_A(x) = \frac{ e_A(x)}{\capac(A) } \leq \frac{D}{ \capac(Q) } , $$
where $D:=\sup_y \deg(y)$.

Now, for any $v_j,v_i \in Q$,
\begin{align*}
g(v_j,v_i) & \leq C \dist(v_j,v_i)^{2-d} \leq C \big| j-i \big|^{2-d} .
\end{align*}
So for $v_i \in Q$, by perhaps modifying the constant $C>0$,
\begin{align*}
\sum_{j=0}^r g(v_i, v_j) & \leq g(v_i,v_i) + C \sum_{j=0}^{i-1} |i-j|^{2-d} + C \sum_{j=i+1}^r |i-j|^{2-d}
\\
& \leq C + C \sum_{j=1}^{i} j^{2-d} + C \sum_{j=1}^{r-i} j^{2-d} .
\end{align*}
Thus,
$$ \sup_{v \in Q} \sum_{u \in Q} g(v,u) \leq
\begin{cases}
\tfrac{C}{3-d} r^{3-d} & \textrm{ if } 2<d<3 , \\
C \log r & \textrm{ if } d = 3 , \\
C & \textrm{ if } d > 3 .
\end{cases}
$$
This enables us to bound the capacity of $Q$ (using \eqref{eqn:capac}):
$$ \capac(Q)  \geq \frac{|Q|}{ \sup_{v \in Q} \sum_{u \in Q} g(v,u) }
\geq
\begin{cases}
c(3-d)  r^{d-2}  & \textrm{ if } 2<d<3 , \\
c \tfrac{r}{\log r} & \textrm{ if } d = 3 , \\
c r & \textrm{ if } d > 3 .
\end{cases}
$$
The proposition follows readily.
\end{proof}

\begin{cor} \label{cor:d dim HK}
Let $G$ be a bounded degree graph rooted at $o$.
Assume that the heat kernel on $G$ satisfies
$$ \sup_{x,y} p_t(x,y) \leq C t^{-d/2}  $$
for some $d>2, C>0$ and all $t > 0$.

Let $(A_t)_t$ be DLA on $G$ started at $A_0 = \{o\}$.
Then, a.s.\
$$ \limsup_{t \to \infty} t^{-\beta} \cdot \rad(A_t) < \infty , $$
where
$$ \beta =
\begin{cases}
\tfrac{1}{d-1}  & \textrm{ if } 2<d<3 , \\
\tfrac12 & \textrm{ if } d > 3 , \\
\end{cases}
$$

When $d=3$ we have that a.s.\
$$ \limsup_{t \to \infty} (t \log t)^{-1/2} \cdot \rad(A_t) < \infty . $$
\end{cor}

\begin{proof}
The $d \neq 3$ cases follow immediately from
combining Proposition \ref{prop:radius Beurling from HK}
with Theorem \ref{thm:kesten bound radius}.

For $d=3$ we have by Proposition \ref{prop:radius Beurling from HK}
that $G$ satisfies a $\vphi$-radius Beurling estimate with $\vphi(r) = C \log r \cdot r^{-1}$.

Choosing $n = \rad(A_s)$ and $t = \lfloor \tfrac{n}{D^2 e^2 \vphi(\rad(A_s) ) } \rfloor$
in Lemma \ref{lem:kesten radius} we obtain for any $r$ that
$$ \Pr [ \tau_{2r} \leq \tau_r + c r^2 (\log r)^{-1} ] \leq e^{-r} , $$
where as before $\tau_r = \inf \{ t \ : \ \rad(A_t) \geq r \}$.
By Borel-Cantelli this implies that a.s.\  there exists $K>0$ such that
for all $r$ we have $\tau_{2r} > \tau_r + K r^2 (\log r )^{-1}$.
This can be translated to say that a.s.\
$$ \limsup_{t \to \infty} (t \log t)^{-1/2} \cdot \rad(A_t) < \infty . $$
\end{proof}

\subsection{Volume Beurling estimates}

\begin{prop} \label{prop:general volume Beurling}
Let $G$ be a bounded degree graph rooted at $o$.
Assume that
$$ \sup_{x,y} p_t(x,y) \leq C t^{-d/2} $$
for some $d>4, C>0$ and all $t>0$.

Then, there exists a constant $C_d>0$ such that for any $\eps>0$,
$G$ satisfies a $\vphi$-volume Beurling estimate with
$$ \vphi(s) = \tfrac{C_d}{\eps} \cdot s^{-1 + \tfrac{2}{d-2} + \eps} . $$
\end{prop}

\begin{proof}
For any finite set $A$ containing $x$ we get that
\begin{align*}
\Pr_x [ X_t \in A ] & = \sum_{y \in A} p_t(x,y) \leq C t^{-d/2} \cdot |A| .
\end{align*}

Let $0< \eps < \tfrac{d}{2} -2$.
If $\Pr_x [ X_t \in A ] \geq t^{-1-\eps}$ then
we get that
$t^{- 1- \eps + d/2} \leq C |A|$, which implies that
\begin{align*}
\sum_{y \in A} g(x,y)
& = \sum_t \Pr_x [ X_t \in A ]
\leq (C |A|)^{\alpha} + \sum_{t > (C|A|)^{\alpha} } t^{-1-\eps}
\leq \eps^{-1} C_{d} |A|^\alpha ,
\end{align*}
where $\alpha = \tfrac{2}{d-2 - 2\eps}$, and
$C_{d} > 0$ is some constant depending on $d$.
Thus,
$$ \capac(A) \geq \frac{ |A| }{ \sup_{x \in A} \sum_{y \in A} g(x,y) }
\geq \eps (C_{d})^{-1} \cdot |A|^{1-\alpha} . $$
The volume Beurling estimate follows from
$h_A(x) = \frac{e_A(x)}{\capac(A) } \leq \frac{D}{\capac(A)}$, and from the fact that
$\alpha = \tfrac{2}{d-2-2\eps} \leq \tfrac{2}{d-2} + \eps$ (because $d>4$).
\end{proof}

\begin{rem}
While this paper was being prepared,
a recent paper \cite{LPS17} was uploaded,
where it is shown (in Theorem 1.2) that in a Cayley graph $G$,
the expected number of time spent by the random walk in $B(o,r)$, a ball of radius $r$,
is bounded by $O(r^2 \log |B(o,r)|) = O(r^3)$.
Thus, under a heat kernel decay as in Proposition \ref{prop:general volume Beurling},
with the same proof
we would obtain a $\vphi$-volume Beurling bound with $\vphi(s) = C s^{-1 + \tfrac3d}$.

We conjecture, as in \cite{LPS17}, that the expected time spent in $B(o,r)$ should be $O(r^2)$
for any Cayley graph.  (This is an interesting open question in its own right.)
This would give a volume Beurling estimate of $O(s^{-1+ \tfrac2d})$.

Since we can obtain the ``correct'' bound in the polynomial growth case
(see the proof of Theorem \ref{thm:poly-growth} below),
and since this does not significantly impact our results for super-polynomial growth groups,
we do not elaborate on these bounds.
\end{rem}

\begin{cor}
\label{cor:volume Beurling}
Let $G$ be a bounded degree graph rooted at $o$.
Assume that
$$ \sup_{x,y} p_t(x,y) \leq C t^{-d/2} $$
for some $d>4, C>0$ and all $t>0$.

Then, there exists a constant $C_d>0$ such that
$G$ satisfies a $\vphi$-volume Beurling estimate with
$$ \vphi(s) = C_d \cdot (\log s) \cdot s^{-1 + \tfrac{2}{d-2} } . $$
\end{cor}

\begin{proof}
By Proposition \ref{prop:general volume Beurling}, if $|A| = s$ then
$h_A(x) \leq \eps^{-1} s^{\eps} \cdot C_d \cdot s^{-1+ \tfrac{2}{d-2} }$
for all $x$.
Taking $\eps = (\log s)^{-1}$ we get that
$$ h_A(x) \leq C_d \cdot e \log s \cdot s^{-1 + \tfrac{2}{d-2} } .$$
\end{proof}

\begin{cor}
\label{cor:high dim HK volume}
Let $G$ be a bounded degree graph rooted at $o$.
Assume that the heat kernel on $G$ satisfies
$$ \sup_{x,y} p_t(x,y) \leq C t^{-d/2}  $$
for some $d>4, C>0$ and all $t > 0$.

Let $(A_t)_t$ be DLA on $G$ started at $A_0 = \{o\}$.
Then, a.s.\ for any $\eps>0$,
$$ \limsup_{t \to \infty} (\log t)^{-1} \cdot t^{-\tfrac{2}{d-2}} \cdot \rad(A_t) = 0 . $$
\end{cor}

\begin{proof}
This is an immediate consequence of
using Corollary \ref{cor:volume Beurling}
to take $\alpha = 1 - \tfrac{2}{d-2}, \beta =1$  in
Theorem \ref{thm:kesten bound volume}.
\end{proof}

\section{Isoperimetric dimension}

For a graph $G$ the \define{isoperimetric dimension} of $G$
is defined as:
$$ \isodim(G) = \inf \set{ d \geq 1  \ | \
\exists \ c>0 \ , \
\forall \ \textrm{ finite } A \subset G \ , \ |\p A | \geq c |A|^{(d-1)/d} } . $$
Here $\p A = \{ x \not\in A  \ : \ \dist(x,A) = 1 \}$.
For example, $\isodim(\Z^d) = d$.

\begin{cor} \label{cor:isodim}
Let $G$ be a bounded degree graph rooted at $o$.
Suppose that $\isodim(G) \geq d$.

Let $(A_t)_t$ be DLA on $G$ started at $A_0 = \{o\}$.
Then, a.s.\
$$ \limsup_{t \to \infty} t^{-\beta} \cdot \rad(A_t) < \infty , $$
where
$$ \beta =
\begin{cases}
\tfrac{1}{d-1}  & \textrm{ if } 2<d<3 , \\
\tfrac12 & \textrm{ if } 3 < d \leq 4 .
\end{cases}
$$
When $d=3$ we have that a.s.\
$$ \limsup_{t \to \infty} (t \log t)^{-1/2} \cdot \rad(A_t) < \infty . $$
For $d > 4$ we have that a.s.\
$$ \limsup_{t \to \infty} (\log t)^{-1} \cdot t^{- \tfrac{2}{d-2}  } \cdot \rad(A_t) < \infty . $$
\end{cor}

\begin{proof}
The method of evolving sets \cite{evolving} (also \cite[Chapter 8]{Gabor}) can be used to show that
if $\isodim(G) \geq d$ then $\sup_{x,y} p_t(x,y) \leq C t^{-d/2}$.  Thus,
the assertion for all $d \leq 4$ cases follows from Corollary \ref{cor:d dim HK},
and for $d>4$ we use Corollary \ref{cor:high dim HK volume}.
\end{proof}

\section{Examples}

\label{scn:examples}

\subsection{Transitive graphs}

\begin{thm} \label{thm:super-poly growth}
Let $G$ be a transitive graph. Fix $o \in G$ and let $B_r$ be the ball of radius $r$ about $o$.
Assume that $G$ has super-polynomial growth; that is, for any $d>0$ we have
$$ \liminf_{r \to \infty} |B_r| r^{-d} = \infty . $$

Let $(A_t)_t$ be DLA on $G$ started at $A_t = \{o\}$.
Then, a.s.\
$$ \forall \ \eps>0 \ , \ \limsup_{t \to \infty} t^{-\eps} \cdot \rad(A_t) = 0 . $$
\end{thm}

\begin{proof}
The Coulhon-Saloff-Coste inequality \cite[Chapter 5]{Gabor} tells us that under the condition
$\liminf_{r \to \infty} |B_r| r^{-d} = \infty$ we have that $\isodim(G) \geq d$.
Thus, the assertion follows from considering $d \to \infty$ in Corollary \ref{cor:isodim}.
\end{proof}

\subsection{Transitive graphs of polynomial growth}

Let $\Gamma$ be a finitely generated group.  Fix some finite,
symmetric generating set $S = S^{-1} , |S| < \infty$.
Let $G$ be the Cayley graph of $\Gamma$ with respect to $S$.
Consider $o=1_\Gamma$ as the root vertex.

A Cayley graph has {\em polynomial growth} if balls grow at most polynomially in the radius;
that is, if $\# \{ x \ : \ |x| \leq r \} \leq C r^d$ for some $C,d>0$ and every $r$.
The implicit constant $C$ depends on the specific choice of Cayley graph, but
the polynomial degree $d$ does not.
Thus, we may speak of a group with polynomial growth.

A famous theorem of Gromov \cite{Gromov}
states that any finitely generated group of polynomial growth
has a finite index subgroup that is nilpotent.
As a consequence, the theory of nilpotent groups tells us that if $\Gamma$ is a finitely generated group
of polynomial growth then there exists an {\em integer} $d>0$, such that
for any Cayley graph $G$ of $\Gamma$, there exists a constant $C>0$ such that
$C^{-1} r^d \leq \# \{ x \ : \ |x| \leq r \} \leq C r^d$ (\ie the graph has {\em uniform} polynomial growth).
Moreover, results in \cite{Trofimov}
imply that any transitive graph $G$ of polynomial growth
also satisfies the above uniform growth estimate (with integer $d$).

\begin{thm} \label{thm:poly-growth}
Let $G$ be a transitive graph of polynomial growth of degree $d > 2$.
Fix $o \in G$ and
let $(A_t)_t$ be DLA on $G$ started at $A_0 = \{o \}$.
Then, when $d \geq 4$, a.s.\
$$ \limsup_{t \to \infty} t^{-2/d}  \cdot \rad(A_t) < \infty , $$
and when $d=3$, a.s.\
$$ \limsup_{t \to \infty}  (t \log t)^{-1/2} \cdot \rad(A_t) < \infty . $$
\end{thm}

\begin{proof}
The Coulhon \& Saloff-Coste inequality \cite{CSC}
(also \cite[Chapter 5]{Gabor}) states that under condition
$C^{-1} r^d \leq \# \{ x \ : \ |x| \leq r \} $ for all $r>0$,
we have that $\isodim(G) \geq d$.

Thus, the $d=3$ case follows from Corollary \ref{cor:isodim}.

For the $d \geq 4$ case, it suffices to prove a $\vphi$-volume Beurling estimate,
with $\vphi(s) = C s^{-1+2/d}$, and plug this into Theorem \ref{thm:kesten bound volume}.

Indeed, the Coulhon \& Saloff-Coste inequality mentioned above gives that $\isodim(G) \geq d$.
Thus, as before, $g(x,y) \leq C (1+\dist(x,y) )^{2-d}$ by the Hebisch \& Saloff-Coste result \cite{HSC}.
This immediately implies that for any finite set $A$ containing $o$,
\begin{align*}
\sum_{y \in A} g(o,y)  & \leq \sum_{r=0}^{\infty} C (1+r)^{2-d} \cdot \# \{ y \in A \ : \ |y| = r \} .
\end{align*}
Since $(1+r)^{2-d}$ is decreasing in $r$, and since
$$ \sum_{r=0}^{\infty}   \# \{ y \in A \ : \ |y| = r \} = |A| , $$
this is maximized if we move more mass to the smaller $r$'s.  That is,
\begin{align*}
\sum_{y \in A} g(o,y)  & \leq C \cdot \sum_{r=0}^{R} (1+r)^{2-d} \# \{ y \ : \ |y| = r \} ,
\end{align*}
where $R$ is the smallest radius such that $|A| \leq  \# \{ y \ : \ |y| \leq R \}$.
Using the upper bound on the polynomial growth we obtain
\begin{align*}
\sum_{y \in A} g(o,y)  & \leq C + C \cdot \sum_{r=1}^{R} r^{2-d} r^{d-1} \leq C R^2 .
\end{align*}
Since $\isodim(G) \geq d$, we know that $\# \{ y \ : \ |y| \leq r \} \geq c r^d$,
for some $c>0$ and all $r$.  Thus, we arrive at
$R^2 \leq C (\# \{ y \ : \ |y| \leq R-1 \} )^{2/d} \leq C |A|^{2/d}$.
Since this holds for any set $A$ containing $o$, by translating, we conclude that
for any finite $A$ and any $x \in A$,
$$ \sum_{y \in A} g(x,y) \leq C |A|^{2/d} . $$
This implies that
$$ \capac(A) \geq C |A|^{1-2/d} , $$
which in turn implies that
$h_A(x) \leq C |A|^{-1+2/d}$.
We conclude that $G$ satisfies a $\vphi$-volume Beurling estimate with
$\vphi(s) =  C s^{-1+2/d}$.

Using this in Theorem \ref{thm:kesten bound volume} we obtain that for $d >3$,
a.s.\
$$ \limsup_{t \to \infty} t^{-2/d} \cdot \rad(A_t) < \infty . $$
\end{proof}

\begin{rem}
It is worth noting that the above proof works for all $d>2$.
For transitive graphs, $d$ must be an integer,
so this is valid for $d=3$ as well.
This is the original bound given in Kesten's papers \cite{Kesten1, Kesten2, Kesten3} for $\Z^d$,
and for $d=3$ it gives a bound of
$$ \limsup_{t \to \infty} t^{-2/3} \cdot \rad(A_t) < \infty . $$

However, in the $d=3$ case,
the bound obtained using the radius Beurling estimate,
namely,
$$ \limsup_{t \to \infty}  (t \log t)^{-1/2} \cdot \rad(A_t) < \infty  $$
is better than the one obtained with the volume estimate.
\end{rem}

\subsection{Perturbations of $\Z^d$}

Consider the following random perturbation of $\Z^d$.
Let every vertex of $\Z^d$ be declared open with probability $p$
and closed with probability $1-p$ independently.
It is well known that there exist some $0< p_c(d) < 1$,
such that when $p>p_c(d)$ this results in a graph containing a unique infinite connected component.
Let this random subgraph of $\Z^d$ be denoted by $\Omega_d$.

The (quenched) heat kernel and volume growth
on this graph have been thoroughly studied.
It is known (see \cite{MR04}) that the heat kernel decays
like in $\Z^d$ itself, for a.e.\ $\Omega_d$ (provided $p>p_c(d)$).
Since $\Omega_d$ is a subgraph of $\Z^d$,
the volume of balls of radius $r$ in this graph grows at most like $O( r^d)$.

With these observations it is straightforward to prove the following.
We omit the proof as it basically follows the same argument as of the proof for the transitive case.

\begin{thm} \label{thm:super-crit perc}
Let $\Omega_d$ be the random graph obtained by taking the infinite component of
super-critical percolation on $\Z^d$,
conditioned on $0 \in \Omega_d$.
Let $(A_t)_t$ be DLA on $\Omega_d$ started at $A_0 = \{0\}$.

Then, for a.e.\ $\Omega_d$ and a.s.\ with respect to the DLA process,
$$ \limsup_{t \to \infty} t^{-2/d}  \cdot \rad(A_t) < \infty , $$
for $d \geq 4$,
and when $d=3$,
$$ \limsup_{t \to \infty}  (t \log t)^{-1/2} \cdot \rad(A_t) < \infty . $$
\end{thm}

\subsection{pre-Sierpinski carpet}

Our methods are robust enough  to lend themselves to more general situations, such as fractal
graphs.  As an example, we consider $\mathcal{S}_n$,
the $n$-dimensional pre-Sierpinski carpet, as defined in \cite{Osada}.
Is is shown there that the heat kernel on $\mathcal{S}_n$ decays like
$$ \sup_{x,y} p_t(x,y) \leq C t^{-d(n)/2} , $$
where
\begin{align}
\label{eqn:dn}
 d(n) & = \frac{ \log (3^n-1) }{ \log (3^n-1) - \log (3^{n-1} - 1) } .
\end{align}
Using this in Corollaries \ref{cor:d dim HK} and  \ref{cor:high dim HK volume}
we can conclude:

\begin{thm} \label{thm:Sierpinski}
Let $\mathcal{S}_n$ be the $n$-dimensional pre-Sierpinski carpet (as defined in \cite{Osada}).
Let $(A_t)_t$ be DLA on $\mathcal{S}_n$, started at $A_0 = \{o\}$,
for some fixed origin $o \in \mathcal{S}_n$.
Then, a.s.\
$$ \limsup_{t \to \infty} t^{-\beta} \cdot \rad(A_t) < \infty , $$
where
$$ \beta =
\begin{cases}
\tfrac{ \log_2 (13)  - 2}{3} = 0.5568\ldots & \textrm{ if } n = 3 , \\
\tfrac12 & \textrm{ if } n = 4 .
\end{cases}
$$
When $n \geq 5$
we have a.s.\
$$ \limsup_{t \to \infty} (\log t)^{-1} t^{-\tfrac{2}{d(n)-2} } \cdot \rad(A_t) < \infty $$
where $d(n)$ is given in \eqref{eqn:dn}.
\end{thm}

\subsection{Graphs of exponential growth}

Let $G$ be a transitive graph of exponential growth;
that is, there exist constants $K>1$ such that
$|B(x,r)| \geq K^r$ for all $r \geq 1$ (and all vertices $x$),
where $B(x,r)$ is the ball of radius $r$ about $x$.
The Coulhon \& Saloff-Coste inequality \cite{CSC} (also \cite[Chapter 5]{Gabor})
tells us that in this case, there exists a constant $c>0$ such that
for any finite set $A$ in $G$ we have
\begin{align} \label{eqn:log iso}
| \p A | & \geq c \cdot |A| \cdot (\log (|A|+2) )^{-1} .
\end{align}

The conclusion \eqref{eqn:log iso} can also be reached in the non-transitive case
under a more restrictive growth condition.
We say that a graph $G$ has \define{pinched exponential growth} if there exist constants $C , K >1$
such that for any vertex $x$ in $G$ and any radius $r$ we have that $B(x,r)$,
the ball of radius $r$ about $x$, is of size $C^{-1} K^r \leq |B(x,r)| \leq C K^r$.
(That is, pinched exponential growth means that all balls grow exponentially
with the same exponent,
up to constants.)
Benjamini \& Schramm showed in \cite{BS03} that if $G$ has pinched exponential growth
then \eqref{eqn:log iso} holds for all finite sets $A \subset G$.

\begin{prop} \label{prop:exp growth Beurling}
Let $G$ be a graph that is either transitive of exponential growth or
non-transitive of pinched exponential growth.
Then, $G$ satisfies a
$\vphi$-volume Beurling estimate with
$\vphi(s) = C s \cdot (\log s)^{-3}$.
\end{prop}

\begin{proof}
By using the method of evolving sets \cite[Chapter 8]{Gabor},
once \eqref{eqn:log iso} holds, we obtain a bound on the heat kernel:
$$ \sup_{x,y} p_t(x,y) \leq C \exp ( - c t^{1/3} ) . $$

Now we continue similarly to the proof of Proposition \ref{prop:general volume Beurling}.
If $\Pr_x [ X_t \in A ] \geq t^{-2}$, then
$$ t^{-2} \leq \Pr_x [ X_t \in A ] \leq |A| \cdot C \exp ( - c t^{1/3} ) , $$
which implies that (perhaps by modifying the constant $C>0$)
$t \leq C (\log |A|)^3$.  This in turn implies that for any $x \in A$,
$$ \sum_{y \in A} g(x,y) \leq \sum_t \Pr_x [ X_t \in A ] \leq C (\log |A|)^3 . $$
Thus,
$\capac(A) \geq |A| \cdot (\log |A|)^{-3}$ and the Beurling estimate follows.
\end{proof}

Using Theorem \ref{thm:kesten bound volume} we now arrive at:

\begin{thm} \label{thm:exp growth}
Let $G$ be a graph that is either transitive of exponential growth or
non-transitive of pinched exponential growth.
Let $o$ be some vertex, and let $(A_t)_t$ be DLA on $G$ started at $A_0 = \{o\}$.

Then, a.s.\
$$ \limsup_{t \to \infty} (\log t)^{-4} \cdot \rad(A_t) < \infty . $$
\end{thm}

\subsection{Non-amenable graphs}

The \define{Cheeger constant} of a graph $G$ is defined to be
$$ \Phi(G) : = \inf_{\substack{A \subset G \\ |A| < \infty } } \frac{|\p A|}{|A|} . $$
A graph $G$ is \define{non-amenable} if $\Phi(G) > 0$.
Note that any non-amenable graph has exponential volume growth, and also
has infinite isoperimetric dimension.
For more on this fundamental notion see \eg \cite{Gabor}.
It is a well known result of Kesten (see \eg \cite[Chapter 7]{Gabor})
that $\Phi>0$ if and only if the bottom of the spectrum of the Laplacian $\Delta = I-P$
is positive.  Here, $P(x,y) = p_1(x,y)$ is the transition matrix of the random walk on $G$,
and the bottom of the spectrum of $\Delta$ is just the infimum over all positive eigenvalues of $\Delta$.

Lemma 2.1 in \cite{BNP} states the following: Let $G$ be a graph, and let $\lambda$
be the bottom of the spectrum of the Laplacian on $G$
(so $\lambda>0$ if and only if $G$ is non-amenable, by Kesten's result).
Then, for any finite subset $A$ of $G$,
\begin{align}
\label{eqn:BNP}
\capac(A) & \geq \lambda \cdot \sum_{x \in A} \deg(x)  .
\end{align}
As a result, we have a volume Beurling estimate for any non-amenable graph $G$:

\begin{prop} \label{prop:non-amen Beurling}
Let $G$ be a non-amenable graph, and let $\lambda>0$ be the bottom of the spectrum
of the Laplacian on $G$.
Then, $G$ satisfies a $\vphi$-volume Beurling estimate with
$\vphi(s) = C (\lambda s)^{-1}$.
\end{prop}

Thus we can also deduce optimal bounds for the growth of DLA on non-amenable graphs.

\begin{thm} \label{thm:non amen DLA}
Let $G$ be a bounded degree non-amenable graph rooted at $o$.
Let $(A_t)_t$ be DLA on $G$ started at $A_0 = \{o\}$.

Then, a.s.\
$$ 0 < \inf_{t > 1}  (\log t)^{-1}  \cdot \rad(A_t) \leq
\limsup_{t \to \infty} (\log t)^{-1} \cdot \rad(A_t) < \infty . $$
\end{thm}

\begin{proof}
Plugging the Beurling estimate in Proposition \ref{prop:non-amen Beurling}
into Theorem \ref{thm:kesten bound volume}, we obtain the upper bound.
The lower bound is a deterministic (non-random) statement.
Since $G$ has exponential growth, $t$ particles must reach distance at least $c \log t$
for sufficiently small constant $c>0$.
\end{proof}

\section{Further questions}

We conclude with some possible questions for further research regarding the structure of the
final DLA aggregate.

\begin{ques}
Let $G$ be a Cayley graph and consider $(A_t)_t$, DLA on $G$ started at $A_0=\{1\}$.
Let $A_\infty = \bigcup_t A_t$ be the final aggregate.

How many topological ends does $A_\infty$ have?

Is there a nice characterization for $A_\infty$ having infinitely many ends, vs.\ finitely many?

Specifically, what about the above questions in the non-amenable setting?  In the hyperbolic setting?
(DLA in the hyperbolic setting was studied in \cite{Eldan}.)
\end{ques}

\begin{ques}
Let $G$ be a non-Liouville Cayley graph (\ie with a non-trivial Poisson boundary).
Consider $(A_t)_t$, DLA on $G$ started at $A_0=\{1\}$,
and let $A_\infty = \bigcup_t A_t$ be the final aggregate.

Is it true that every point in the Poisson boundary is a limit of a sequence of vertices occupied
by $A_\infty$?
(For example, on the free group (regular tree) this is not difficult to show.)
\end{ques}

\end{document}